\documentclass[12pt]{amsart}
\usepackage[top=1in, bottom=1in, left=1in, right=1in]{geometry}

\usepackage{amssymb,amsmath,amsthm}
\usepackage{mathrsfs} 
\usepackage{bbm} 
\usepackage{enumitem}
\usepackage[hyphens]{url}

\usepackage{graphicx}
\usepackage{color}
\usepackage{pgfplots}

\numberwithin{equation}{section}

\theoremstyle{plain}
\newtheorem{thm}{Theorem}[section]
\newtheorem{lem}[thm]{Lemma}

\newtheorem{cor}[thm]{Corollary}

\theoremstyle{remark}
\newtheorem{rem}[thm]{Remark}

\newcommand{\N}{\mathbb{N}}
\newcommand{\Z}{\mathbb{Z}}

\newcommand{\R}{\mathbb{R}}
\newcommand{\C}{\mathbb{C}}

\newcommand{\F}{\mathbb{F}}

\newcommand{\CE}{\mathcal{E}}

\newcommand{\CJ}{\mathcal{J}}

\newcommand{\CL}{\mathcal{L}}

\newcommand{\CP}{\mathcal{P}}

\newcommand{\CR}{\mathcal{R}}

\newcommand{\CZ}{\mathcal{Z}}

\newcommand{\abs}[1]{\left\lvert #1 \right\rvert}

\newcommand{\bs}\boldsymbol{}

\newcommand{\eq}[2]{ \begin{equation} \label{#1}\begin{split} #2 \end{split} \end{equation} }
\newcommand{\al}[1]{\begin{align} #1 \end{align} }
\newcommand{\als}[1]{\begin{align*} #1 \end{align*} }

\newcommand{\sub}[1]{ \substack{ #1 } } 
\newcommand{\ssum}[1]{\,\sideset{}{^*}\sum_{ #1 } } 

\newcommand{\nn}{\nonumber \\}

\newcommand{\dee}{\mathrm{d}}

\renewcommand{\bar}[1]{\overline{#1}}

\renewcommand{\mod}[1]{\,({\rm mod}\,#1)}

\pgfplotsset{my style/.append style={axis x line=bottom, axis y line=
left, xlabel={$\theta$}, ylabel={$\alpha$}, axis equal }}

\begin{document}

\title{Primes in short arithmetic progressions}

\author{Dimitris Koukoulopoulos}
\address{D\'epartement de math\'ematiques et de statistique\\
Universit\'e de Montr\'eal\\
CP 6128 succ. Centre-Ville\\
Montr\'eal, QC H3C 3J7\\
Canada}
\email{{\tt koukoulo@dms.umontreal.ca}}


\date{\today}

\maketitle

\begin{center}
\dedicatory{\textit{ \small{In the memory of Paul and Felice Bateman, and of Heini Halberstam}}}
\end{center}

\begin{abstract} Let $x,h$ and $Q$ be three parameters. We show that, for most moduli $q\le Q$ and for most positive real numbers $y\le x$, every reduced arithmetic progression $a\mod q$ has approximately the expected number of primes $p$ from the interval $(y,y+h]$, provided that $h>x^{1/6+\epsilon}$ and $Q$ satisfies appropriate bounds in terms of $h$ and $x$. Moreover, we prove that, for most moduli $q\le Q$ and for most positive real numbers $y\le x$, there is at least one prime $p\in(y,y+h]$ lying in every reduced arithmetic progression $a\mod q$, provided that $1\le Q^2\le h/x^{1/15+\epsilon}$.
\end{abstract}


\section{Introduction}
Let
\[
E(x,h;q) = \max_{(a,q)=1}  \abs{ \sum_{\substack{x<p\le x+h \\ p\equiv a\mod{q}} } \log p -  \frac{h}{\phi(q)}  }  .
\]
It is believed that $E(x,h;q)\ll_\epsilon x^\epsilon \sqrt{h/q}$ for all $1\le q\le h\le x$, which would imply that each subinterval of $(x,2x]$ of length $>qx^{\epsilon'}$ has its fair share of primes from each reduced arithmetic progression $a\mod q$ (see \cite{Mon;1976,FG;1989} for results and conjectures concerning the case $h=x$). Proving such a result lies well beyond the current technology. However, for several applications it turns out that bounding $E(x,h;q)$ {\it on average} suffices. The case $h=x$, is a rephrasing of the famous Bombieri-Vinogradov theorem: for each fixed $A>0$, there is some $B=B(A)>0$ such that
\[
\sum_{q\le x^{1/2}/(\log x)^B} E(x,x;q) \ll \frac{x}{(\log x)^A} .
\]
Subsequently, various authors focused on obtaining similar results for $h$ small compared to $x$. The first such results were obtained by Jutila \cite{Jut;1969}, Motohashi \cite{Mot;1971}, and Huxley and Iwaniec \cite{HI;1975}. Their bounds were subsequently improved by Perelli, Pintz and Salerno \cite{PPS;1984,PPS;1985} and, finally, by Timofeev \cite{Tim;1987}, who showed that
\eq{bv-short}{
\sum_{q\le Q} E(x,h;q) \ll \frac{h}{(\log x)^A} 
}
when $x^{3/5}(\log x)^{2A+129}\le y\le x$ and $Q\le y/(\sqrt{x}(\log x)^{A+64})$, and when $x^{7/12+\epsilon}\le y\le x$ and $Q\le y/x^{11/20+\epsilon}$ with $\epsilon$ fixed and positive.

Results such as the above ones are closely related to what we call {\it zero-density estimates}. Typically, such an estimate is an inequality of the form 
\eq{zero-density}{
\sum_{q\le Q} \ssum{\chi\mod q}N(\sigma,T,\chi)\ll (Q^2T)^{c(1-\sigma)} \log^M(QT) 
	\qquad(Q\ge1,\ T\ge2,\ 1/2\le\sigma\le 1) ,
}
where $c$ and $M$ are some fixed numbers, $N(\sigma,T,\chi)$ denotes the number of zeroes $\rho=\beta+i\gamma$ of the Dirichlet $L$-function $L(s,\chi)$ with $\beta\ge\sigma$ and $|\gamma|\le T$, and the symbol $\ssum{}$ means that we are summing over primitive characters only. The best result of this form we currently know is with $c=12/5+\epsilon$ (and $M=14$ is admissible), a consequence of \cite[Theorem 12.2, eqn. (12.13)]{Mon;1971} when $1/2\le\sigma\le3/4$, of \cite[eqn. (1.1)]{Hux;1974} for $3/4\le\sigma\le 5/6$, and of \cite[Theorem 12.2, eqn. (12.14)]{Mon;1971} when $5/6\le\sigma\le1$. The case $c=2$ and $M=1$ is called the {\it Grand Density Conjecture} \cite[p. 250]{IK;2004}, which, for practical purposes, is often as strong as the Generalized Riemann Hypothesis itself. Proving that \eqref{zero-density} holds for some $c<12/5$ would immediately imply relation \eqref{bv-short} in a wider range of $h$.

\medskip

In the present paper, we study the distribution of primes in short arithmetic progressions too, with the difference that we let the exact location of the interval $(x,x+h]$ vary. The first result of this flavour was shown by Selberg \cite{Sel;1943} when $Q=1$, whose work implies that
\eq{Q=1}{
\int_x^{2x} \abs{\sum_{y<p\le y+h} \log p - h} \dee y \ll \frac{hx}{(\log x)^A}
}
for all $A>0$, as long as $h>x^{19/77+\epsilon}$. Huxley's results \cite{Hux;1972} allows one to demonstrate \eqref{Q=1} when $h>x^{1/6+\epsilon}$. Finally, if relation \eqref{zero-density} is true for $Q=1$ and some $c\ge2$, then \eqref{Q=1} holds for $h>x^{1-2/c+\epsilon}$ (see, for example, \cite[Exercise 5]{IK;2004}). Our first result is a generalization of this statement.


\begin{thm}\label{mainthm1}
Assume that relation \eqref{zero-density} holds for some $c\in[2,4]$. Fix $A\ge1$ and $\epsilon\in(0,1/3]$. If $x\ge h\ge1$ and $1\le Q^2\le h/x^{1-2/c+\epsilon}$, then
\[
\int_x^{2x}\sum_{q\le Q}  E(y,h;q) \dee y   \ll_{\epsilon,A} \frac{hx}{(\log x)^A}  .
\]
\end{thm}


Theorem \ref{mainthm1} will be proven in Section \ref{mainthm1-proof}. The fact that $h$ is allowed to cover a longer range compared to the case when the interval $(x,x+h]$ is fixed was important in \cite{CDKS2} and \cite{CDKS1}, where primes in progressions were needed to be found in intervals $(x,x+h]$ of length $h\asymp \sqrt{x}$.

Letting $c=12/5+\epsilon$ in Theorem \ref{mainthm1} allows us to take $Q^2=h/x^{1/6+\epsilon'}$. Using a more sophisticated approach based on the frequency of large values of Dirichlet polynomials due to Gallagher-Montgomery and Huxley, it is possible to improve this result when $h>\sqrt{x}$. This is the context of the next theorem, which will be proven in Section \ref{mainthm2-proof}. Note that when $h$ approaches $x$, our result converges towards a weak version of the Bombieri-Vinogradov theorem.


\begin{thm}\label{mainthm2} Fix $A\ge1$ and $\epsilon\in(0,1/3]$. Let $x\ge1$, $h=x^\theta$ with $1/6+2\epsilon\le \theta\le 1$, and $Q\ge 1$ such that $Q^2\le h/x^{\alpha+\epsilon}$, where
\[
\alpha = \begin{cases}
		(1-\theta)/3			&\text{if}\ 5/8\le \theta\le 1,  \\
		1/8  					&\text{if}\ 13/24 \le \theta\le 5/8, \\
		2/3-\theta   			&\text{if}\ 1/2\le \theta\le 13/24 ,\\
		1/6					&\text{if}\ 1/6+2\epsilon\le \theta\le 1/2.
\end{cases}
\]
Then we have that
\[
\int_x^{2x}\sum_{q\le Q}  E(y,h;q) \dee y   \ll_{\epsilon,A} \frac{hx}{(\log x)^A}  .
\]
\end{thm}


\noindent The graph of $\alpha$ as a function of $\theta$ is given below.

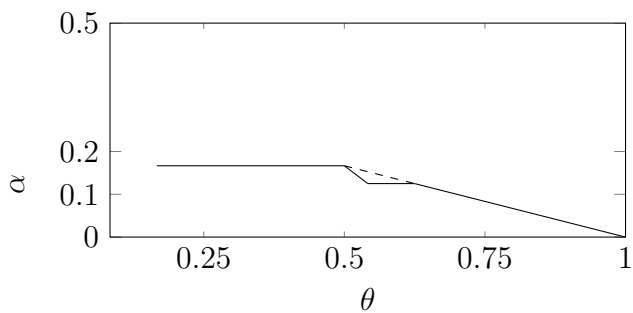
\begin{center}
\begin{tikzpicture}
  \begin{axis}[xlabel=$\theta$, ylabel=$\alpha$, xmax=1, ymin=0, ymax=0.5, 
  	xtick={0,0.25, 0.5, 0.75, 1}, ytick={0, 0.1, 0.2, 0.5}, xscale=1,yscale=0.5]
  
    \addplot[domain=1/6:1/2]{1/6};
    \addplot[domain=1/2:13/24]{2/3-x};
    \addplot[domain=13/24:5/8]{1/8};
    \addplot[domain=5/8:1]{(1-x)/3};
     \addplot[domain=1/2:5/8,dashed]{(1-x)/3};
  \end{axis}
\end{tikzpicture}
\end{center}


Finally, we note that using some results due to Li \cite{Li;1997} allows us to take $\alpha=1/15$ in Theorem \ref{mainthm2} if we contend ourselves with only lower bounds on the number of primes in a short arithmetic progressions. A related result was proven by Kumchev \cite{Kum;2002} but for a fixed short interval. It should be noted that when $Q=1$, Jia \cite{Jia;1996} showed that $\alpha=1/20$ is admissible. However, the proof of Jia's =result uses some more specialized results concerning Kloosterman sums that do not have exact analogs when we add a long average over arithmetic progressions as well (though in \cite{HWW;2004} a related result was proven).


\begin{thm}\label{mainthm3} Fix $\epsilon$ and consider $x\ge h\ge 2$ and $1\le Q^2\le h/x^{1/15+\epsilon}$. 
Then there is a constant $c=c(\epsilon)>0$ such that, for all $A>0$, we have
\[
\#\left\{  \begin{array}{c} (q,n)\in \N^2 \\ q\le Q,\ n\le x \end{array}   :
	 \sum_{\substack{n<p\le n+h \\ p\equiv a\mod q}}\log p\ge \frac{ch}{\phi(q)}\quad\text{when}\ (a,q)=1\right\} 
	= Qx + O\left(\frac{Qx}{(\log x)^A}\right) .
\]
\end{thm}


The proof of this result, which will be given in Section \ref{mainthm3-proof}, will be relatively short as we will almost immediately appeal to Li's results and methods from \cite{Li;1997}. We state it and prove it also because of an interesting application it has to a rather distant problem studied in \cite{BPS;2012} and in \cite{CDKS1}. There the quantity of interest was $S(M,K)$, which is defined to be the number of pairs $(m,k)\in \N^2$ with $m\le M$ and $k\le K$ for which there exists an elliptic curve $E$ over $\F_p$ with group of points $E(\F_p)\cong \Z/m\Z\times\Z/mk\Z$. As it was shown by Banks, Shparlinski and Pappalardi in \cite{BPS;2012}, we have that
\[
S(M,K) = \#\{m\le M,\ k\le K:\exists\ p\in(m^2k-2m\sqrt{k}+1,m^2+2m\sqrt{k}+1),\ p\equiv 1\mod{m}\}.
\]
So a straightforward application of Theorem \ref{mainthm3} implies the following result, which is a strengthening of the unconditional part of Theorem 1.5 in \cite{CDKS1}. 


\begin{cor} Fix $\epsilon>0$ and $A>0$. If $M\le K^{13/34-\epsilon}$, then we have that
\[
S(M,K) = MK+O_{A,\epsilon}\left(\frac{MK}{(\log K)^A} \right) .
\]
\end{cor}


\subsection*{Acknowledgements} I would like to thank Sandro Bettin for many useful discussions around Lemma \ref{charsums} and the use of a smooth partition of unity, and Kaisa Matom\"aki for various useful suggestions and for providing many useful references. I would also like to thank my coauthors in \cite{CDKS2,CDKS1}, Vorrapan Chandee, Chantal David and Ethan Smith, for their helpful comments as well as for their encouragement.

This work was partially supported by the Natural Sciences and Engineering Research Council of Canada Discovery Grant 435272-2013.


\section{The proof of Theorem \ref{mainthm1}}\label{mainthm1-proof}

Let $x,h, Q, \epsilon$ and $A$ as in the statement of the first part of Theorem \ref{mainthm1}. We may assume that $x$ is large enough in terms of $A$ and $\epsilon$. Throughout the rest of the paper, we set
\[
\CL = \log x.
\]
Since $h\ge x^{1-2/c+\epsilon}$ by assumption, relation \eqref{Q=1} and the comments following it imply that it is enough to prove that
\eq{mainthm1-alt1}{
\int_x^{2x} \sum_{q\le Q} E'(y,h;q) \dee y \ll_{A,\epsilon} \frac{hx}{\CL^A} ,
}
where
\[
E'(x,h;q) := \max_{(a,q)=1} \left|  \sum_{\substack{ x < p \le x + h  \\  p\equiv a\mod q}} \log p  
	- \frac{1}{\phi(q)} \sum_{\substack{x<p\le x+h \\(p,q)=1 }}\log p\right| .
\]
(Notice that the condition $(p,q)=1$ in the second sum is trivially satisfied for primes $p>Q$.) 

We further reduce this relation to the bound
\eq{mainthm1-alt2}{
\int_x^{4x}\sum_{q\le Q}E''(y,\eta y;q)\dee y
	\ll_{\epsilon,A}  \frac{\eta x^2}{\CL^A} ,
}
where
\[
\eta := \frac{h}{x\CL^{A+1}} = x^{\theta-1} \CL^{-A-1}.
\] 
Indeed, covering $(y,y+h]$ by intervals of the form $(y_j,y_{j+1}]$, where $y_j=(1+\eta)^jy$, implies that
\als{
E'(y,h;q) &\le  \sum_{1\le (1+\eta)^j\le 1+h/y} E'(y_j,\eta y_j;q)
		+  O\left(\frac{\eta x}{\phi(q)}\right) ,
}
by the Brun-Titchmarsch inequality. So
\als{
\int_{x}^{2x} \sum_{q\le Q}
	E'(y,h;q) \dee y 
	&\le \sum_{1\le(1+\eta)^j\le 1+h/x} \frac{1}{(1+\eta)^j} \int_{(1+\eta)^jx}^{2(1+\eta)^jx} 
		\sum_{q\le Q}  E'(y,\eta y;q)  \dee y
		+ O(\eta x^2\CL ) \\
	&\ll  \frac{h}{\eta x}\int_x^{4x}\sum_{q\le Q}E'(y,\eta y;q) \dee y
	+   \eta x^2 \CL, 
}
which implies Theorem \ref{mainthm1} if relation \eqref{mainthm1-alt1} holds.  So from now on we focus on proving this relation.

Moreover, we need to remove certain `bad' moduli from our sum. Set
\[
\sigma_0= 1- \frac{c_0\log\CL}{\CL} ,
\]
where $c_0$ is some large constant to be determined later and let $\mathcal{E}$ be the set of moduli $q\le Q$ which are multiples of integers $d$ modulo which there exists a primitive Dirichlet character $\chi$ such that $N(\sigma_0,x,\chi)\ge1$. Note that any such $d$ must be greater than $D_1:=\CL^{4c_0+A+M+3}$, where $M$ is the constant in relation \eqref{zero-density}. This a consequence of the Korobov-Vinogradov zero-free region for $L(s,\chi)$ (see the notes of Chapter 9 in \cite{Mon;1994}) and of Siegel's theorem \cite[p. 126]{Dav;2000}. Therefore
\als{
\int_x^{4x}  \sum_{q\in\CE }  E'(y,\eta y;q)  \dee y
	\ll  \sum_{q\in\CE  }\frac{\eta x^2}{\phi(q)}
	& \le  \eta x^2\CL \sum_{D_1\le d\le Q}\ \sideset{}{^*}\sum_{\chi\mod d} N(\sigma_0,x,\chi)
		\sum_{q\le Q,\,d|q}\frac1{q} \\
	&\ll  \eta x^2\CL^2 \sum_{D_1\le d\le Q}\frac1{d} 
		\ \sideset{}{^*}\sum_{\chi\mod d}N(\sigma_0,x,\chi)
}
by the Brun-Titchmarsch inequality. Splitting the range of $d$ into $O(\CL)$ intervals of the form $[D,2D]$ and applying relation \eqref{zero-density} with $c\le 4$ to each one of them , we find that
\eq{exceptional}{
\int_x^{4x} \sum_{q\in \CE } E'(y,\eta y;q) \dee y
	\ll \eta x^2\CL^{M+2} \max_{D_1 \le D\le 2Q} \frac{(D^2x)^{4(1-\sigma_0)}}{D}
	 \ll \frac{\eta x^{2+4(1-\sigma_0) }}{ \CL^{4c_0+A} }  = \frac{\eta x^2}{\CL^A} .
}
Therefore, instead of \eqref{mainthm1-alt2}, it suffices to show
\eq{mainthm1-alt2b}{
\int_x^{4x}\sum_{\substack{q\le Q \\ q\notin \CE }} E'(y,\eta y;q)\dee y
	\ll_{\epsilon,A}  \frac{\eta x^2}{\CL^A} ,
}

Next, let $\Lambda(n)$ be the von Mangoldt function, defined to be $\log p$ if $n=p^k$ for some $k$, and 0 otherwise, and set
\[
E''(x,h;q) = \max_{(a,q)=1} \left|  \sum_{\substack{ x < n \le x + h  \\  n\equiv a\mod q}} \Lambda(n)  
	- \frac{1}{\phi(q)} \sum_{\substack{y<n\le y+h \\ (n,q)=1}} \Lambda(n) \right|.
\]
Then we have that
\[
E'(y,\eta y;q) - E''(y,\eta y,q) 
	\ll \sum_{\substack{y<p^k\le y+\eta y \\ k\ge2}} \log p 
	 \ll (\eta\sqrt{x}+1)\CL^2
\]
for every $y\in[x,2x]$. Since $Q^2\le \eta x^{1-\epsilon/2} \le x^{1-\epsilon/2}$, we deduce that
\als{
\int_x^{2x} \sum_{q\le Q} |E'(y,\eta y;q) - E''(y,\eta y;q)| \dee y \ll (\eta x^{3/2}Q+xQ)\CL^2 \ll_{\epsilon,A} \frac{\eta x^2}{\CL^A},
}
thus reducing relation \eqref{mainthm1-alt2b} to showing that
\eq{mainthm1-alt3}{
\int_x^{2x} \sum_{\substack{q\le Q \\ q\notin\CE }} E''(y,\eta y;q) \dee y \ll_{A,\epsilon} \frac{\eta x^2}{\CL^A} .
}

The next step is to switch from arithmetic progressions to sums involving Dirichlet characters. Given an arithmetic function $f:\N\to\C$, we set
\[
S(y,h;f)= \left|\sum_{y<n\le y+h} f(n)\right| .
\]
If $\chi$ is induced by $\chi_1$, then we have that
\eq{(n,q)>1}{
\left| S(y,\eta y;\Lambda\chi) - S(y,\eta y;\Lambda\chi_1) \right|
	\le \sum_{ \substack{y< n\le y(1+\eta) \\ (n,q)>1 } } \Lambda(n) 
	\le \omega(q)\log(2y)\ll \CL^2,
}
uniformly in $q\le x$ and $y\le3x$. So for such choices of $q$ and $y$, we have that
\eq{e10}{
E''(y,\eta y;q)
	\le  \frac1{\phi(q)}\sum_{\substack{\chi\mod q \\ \chi\neq\chi_0}}  S(y,\eta y;\chi)
	=   \frac1{\phi(q)}\sum_{d|q,\,d>1}\ \sideset{}{^*}\sum_{\chi\mod d} S(y,\eta y;\chi)
		+  O(\CL^2).
}
Therefore, using the inequality $\phi(dm)\ge\phi(d)\phi(m)$, we find that
\als{
\sum_{\substack{q\le Q \\ q\notin\CE}} E''(y,\eta y;q) 
	&\le  \sum_{\substack{q\le Q \\ q\notin\CE}} \frac1{\phi(q)}\sum_{d|q,\,d>1}
		\ssum{\chi\mod d} S(y,\eta y;\Lambda\chi)
		+ O(Q\CL^2) \\
	&\le \sum_{\substack{1<d\le Q\\d\notin\CE}} \ssum{\chi\mod d} S(y,\eta y;\Lambda\chi) 
		 \sum_{\substack{q\le Q \\ d|q}} \frac1{\phi(q)}
		+ O(Q\CL^2) \\
	&\ll  \CL \sum_{\substack{1<d\le Q\\d\notin\CE}} \frac{1}{\phi(d)} \ssum{\chi\mod d} S(y,\eta y;\Lambda\chi)  + O(Q\CL^2)  .
}
Hence breaking the interval $(1,Q]$ into $O(\CL)$ dyadic intervals of the form $(D,2D]$ implies that
\[
\sum_{q\le Q} E''(y,\eta y;q) 
	\ll \CL^3 \max_{1\le D\le Q} \frac{1}{D} \sum_{\substack{D<d\le 2D\\d\notin\CE}} 
		\ssum{\chi\mod d} S(y,\eta y;\Lambda\chi) 
	 + O(Q\CL^2)  ,
\]
thus reducing relation \eqref{mainthm1-alt3} to showing that
\eq{mainthm1-alt4}{
\int_x^{4x} \sum_{\substack{D<d\le 2D\\d\notin\CE}}  \ssum{\chi\mod d} S(y,\eta y;\Lambda\chi) \dee y
 		\ll \frac{D\eta x^2}{\CL^{A+3}}     \qquad(1\le D\le Q) .
}

In order to prove \eqref{mainthm1-alt4}, we express $S(y,\eta y;\Lambda \chi)$ as a sum over zeroes of $L(s,\chi)$. Let $T_0$ be the unique number of the form $2^j-1$, $j\in\N$, lying in $(x/2,x]$. Applying \cite[p. 118, eqn. (9)]{Dav;2000}, we find that for $q\le x$
\als{
S(y,\eta y;\Lambda \chi)
	&\le \left| \sum_{\substack { \rho=\beta+i\gamma\neq0\,:\, L(\rho,\chi)=0 \\ 0\le\beta\le1,\,|\gamma|\le T_0}}
		\frac{(1+\eta)^\rho-1}{\rho} \cdot y^\rho \right|
			+  O(\CL^2)   \\
	&\le \sum_{1\le2^j\le x/2} \left| \sum_{\substack { \rho=\beta+i\gamma\neq0\,:\, L(\rho,\chi)=0
							\\ 0\le\beta\le1,\,2^j\le|\gamma|+1\le2^{j+1}}}
		\frac{(1+\eta)^\rho-1}{\rho} \cdot y^\rho \right|
			+  O(\CL^2)
}
Now, if $L(s,\chi)$ has no zeroes $\rho=\beta+i\gamma$ with $\beta\ge\sigma_0$ and $|\gamma|\le x$, then the functional equation implies that there are no zeroes with $0\le\beta\le1-\sigma_0$ and $|\gamma|\le x$ (except possibly for $\rho=0$, which excluded from our sum). So if we let
\[
\CZ(T,\chi)
	=  \{\rho=\beta+i\gamma:L(\rho,\chi)=0,\,1-\sigma_0\le\beta\le\sigma_0,\,T\le|\gamma|+1\le2T\},
\]
then we find that
\als{
&\int_x^{4x}\sum_{\substack{D<d\le 2D \\ d\notin\CE }} \ssum{\chi\mod d}
		S(y,\eta y;\Lambda \chi ) \dee y \\
	&\quad\ll \CL \max_{1\le T\le x} \int_x^{4x}  \sum_{D<d\le 2D} \ssum{\chi\mod d}
			\left| \sum_{\rho\in \CZ(T,\chi)} \frac{(1+\eta)^\rho-1}{\rho} \cdot y^\rho \right|  \dee y 
			+ D^2x\CL^2 .
}
The second error term is $\ll D\eta x^2/\CL^{A+3}$, since $D\le Q\le \sqrt{\eta x}/x^{\epsilon/2}$. In order to treat the first error term, we apply the Cauchy-Schwarz inequality. This reduces \eqref{mainthm1-alt4} to showing that
\eq{main 1}{
R(D,T):=\int_x^{4x} \sum_{D<d\le 2D}\ \sideset{}{^*}
	\sum_{\chi\mod d} \left| \sum_{\rho\in \CZ(T,\chi)} \frac{(1+\eta)^\rho-1}{\rho} \cdot y^\rho \right|^2 \dee y
	\ll\frac{\eta^2x^3}{\CL^{2A+8}}
}
for $1\le D^2\le\eta x^{2/c-\epsilon/2}$ and $1\le T\le x$. Using the identity $|z|^2=z\overline{z}$ to expand the square of the absolute value in~\eqref{main 1} and then integrating over $y\in[x,4x]$, we find that
\als{
R(D,T)
	&=\sum_{d\le D}\ \sideset{}{^*}\sum_{\chi\mod d}
		\sum_{\rho_1,\rho_2\in \CZ(T,\chi)}
			 \frac{(1+\eta)^{\rho_1}-1}{\rho_1}\frac{(1+\eta)^{\overline{\rho_2}}-1}{\overline{\rho_2}}
			\cdot \frac{(4x)^{\rho_1+\overline{\rho_2}+1}
			-  x^{\rho_1+\overline{\rho_2}+1}}{\rho_1+\overline{\rho_2}+1}  \\
	&\ll x^3 \CL^2 \min\left\{\eta,\frac{1}{T}\right\}^2   \sum_{d\le D}\ \sideset{}{^*}	
		\sum_{\chi\mod d} \sum_{\rho_1,\rho_2\in \CZ(T,\chi)}
			\frac{x^{\beta_1+\beta_2-2}}{1+|\gamma_1-\gamma_2|}   ,
}
where we have written $\rho_j=\beta_j+i\gamma_j$ for $j\in\{1,2\}$. Since $x^{\beta_1+\beta_2}\le x^{2\beta_1}+x^{2\beta_2}$ and 
\[
\sum_{\rho_j\in \CZ(T,\chi)}\frac{1}{1+|\gamma_1-\gamma_2|}\ll \CL^2  \quad(j\in\{1,2\}) 
\]
(see for example \cite[p. 98, eqn (1) and (2)]{Dav;2000}), we deduce that
\[
R(D,T)
	\ll x^3 \CL^4   \min\left\{\eta,\frac{1}{T}\right\}^2
		\sum_{d\le D}\ \sideset{}{^*}
		\sum_{\chi\mod d} \sum_{\rho\in \CZ(T,\chi)} x^{2\beta-2} .
\]
This reduces \eqref{main 1} to proving that
\eq{main 2}{
R'(D,T):=\sum_{d\le D}\ \sideset{}{^*}
		\sum_{\chi\mod d} \sum_{\rho\in \CZ(T,\chi)} x^{2\beta-2}
	\ll \frac{\max\{\eta T,1\}^2}{\CL^{2A+12}}.
}
For each Dirichlet character $\chi\mod d$, there are at most $O(T\log(dT))$ zeroes $\rho\in\CZ(T,\chi)$ with $\beta\le1/2+1/\CL$ by \cite[p. 101, eq. (1)]{Dav;2000}. For the rest of the zeroes, note that $x^{2\beta-2}\ll \CL\int_{1/2}^\beta x^{2(\sigma-1)}d\sigma$ and consequently \eqref{zero-density} implies that
\als{
R'(D,T)
	&\ll \frac{D^2T\CL}{x} + \CL  \int_{1/2}^{\sigma_0}
		\sum_{d\le D}\ \sideset{}{^*}\sum_{\chi\mod d} N(\sigma,T,\chi)  \frac{\dee \sigma}{x^{2(1-\sigma)}}  \\
	&\ll \frac{D^2T\CL}{x} +   \CL^{M+1} \int_{1/2}^{\sigma_0}
		\left( \frac{ (D^2T)^c } {x^2} \right)^{1-\sigma} \dee\sigma.
}
Since $c\in[2,4]$ and $D^2/\eta\le x^{2/c-\epsilon/2}$, we conclude that
\als{
\frac{R'(D,T)}{\max\{1,\eta T\}^2}
	&\ll\frac{D^2\CL}{\eta x} + \CL^{M+1} \int_{1/2}^{\sigma_0}\left( \frac{ (D^2/\eta)^c} {x^2} \right)^{1-\sigma}\dee\sigma \\
	&\ll\frac{\CL}{x^{1-2/c+\epsilon/2}} + \frac{\CL^{M+2}}{x^{c\epsilon(1-\sigma_0)/2}} 
		\ll \frac{1}{\CL^{c_0c\epsilon/2-M-2}} .
}
Choosing $c_0=2(2A+M+14)/(c\epsilon)$ then proves \eqref{main 2} thus completing the proof of Theorem \ref{mainthm1}.


\section{Some auxiliary results}\label{aux}

Before we embark on the main part of the proof of Theorems \ref{mainthm2} and \ref{mainthm3}, we state here the main technical tools we will use. The first one is a result due to Montgomery and Gallagher\ \cite[Theorem 7.1]{Mon;1971}.


\begin{lem}\label{large sieve} 
Let $\{a_n\}_{n=1}^N$ be a sequence of complex numbers. For $Q\ge1$ and $T\ge1$ we have that
\[
\sum_{q\le Q}\ssum{\chi\mod q}
		\int_{-T}^T  \abs{ \sum_{n=1}^N\frac{a_n\chi(n)}{n^{it}} }^2\dee t
	\ll  (Q^2T+N)   \sum_{n=1}^N|a_n|^2.
\]
\end{lem}


Combining Lemma \ref{large sieve} with a result due to Huxley \cite[Theorem 9.18, p.247]{IK;2004}, we have the following estimate on the frequency of large values of Dirichlets polynomials twisted by Dirichlet characters. Here and for the rest of the paper, given a set
\[
\CR\subset\{(t,\chi): t\in\R,\ \chi\ \text{is a Dirichlet character}\},
\]
we say that $\CR$ is well-spaced if for each $(t,\chi),(t',\chi)\in\CR$ with $t\neq t'$, we have that $|t-t'|\ge1$. Moreover, we let $\tau_m$ denote the number of ways to write $n$ as a product of $m$ positive integers.


\begin{lem}\label{large values} Fix $m\in\N$ and $r\ge0$ and let $\{a_n\}_{n=1}^N$ be a sequence of complex numbers such that $|a_n|\le \tau_m(n)(\log n)^r$ for all $n\le N$. For each Dirichlet character $\chi$, we set $A(s,\chi) = \sum_{n=1}^N a_n\chi(n)/n^s$ and we consider a well spaced set 
\[
\CR \subset \bigcup_{q\le Q}\bigcup_{\substack{\chi\mod q \\ \chi\ \text{primitive}}} 
	\left\{(t,\chi): t\in\R,\    \abs{ A(1/2+it ,\chi )} \ge U \right\} ,
\]
where $U\ge1$, $Q\ge1$ and $T\ge1$ are some parameters. If $H=Q^2T$, then
\[
|\CR| \ll_{m,r} \min\left\{ \frac{N+H}{U^2}, \frac{N}{U^2}+\frac{NH}{U^6} \right\} (\log 2N)^{3m^2+6r+18} .
\]
\end{lem}


Finally, we need the following result which allows us to pass from a sum of characters of length $N$ to a shorter sum when $N$ is large enough. Its proof is a standard application of ideas related to the approximate functional equation of $L$-functions (for example, see Section 9.6 in \cite{IK;2004}).


\begin{lem}\label{charsums}
Let $\chi$ be a primitive Dirichlet character modulo $q\in(1,Q]$, $g:[0,+\infty)\to[0,+\infty)$ be a smooth function supported on $[1,4]$, $t\in\R$, $N\ge1$ and $r\in\Z_{\ge0}$. If $|t|\le T$ for some $T\ge2$, and $M=\max\{1,(QT/N)^{1+\delta}\}$ for some fixed $\delta>0$, then
\[
\sum_{n=1}^\infty \frac{g(n/N)\chi(n)(\log n)^r}{n^{1/2+it}} 
	\ll_{r,\delta} (\log2N)^r
			\int_{-\infty}^{\infty} 
				\left|\sum_{n\le M} \frac{\chi(n)}{n^{1/2+i(u+t)}}\right|
				\frac{\dee u}{1+u^2}  .
\]
\end{lem}


\begin{proof} It suffices to consider the case $r=0$. Indeed, for the general case, note that
\[
(\log n)^r = (\log N+ \log(n/N))^r  = \sum_{j=0}^r \binom{r}{j} (\log N)^j (\log(n/N))^{r-j}.
\]
So the general case follows by the case $r=0$ applied with $g\log^{r-j}$, $0\le j\le r$, in place of $g$. Moreover, we may assume that $\delta\le 1/2$. 

Set $s_0=1/2+it$ and let $\hat{g}(w)=\int_0^\infty g(u)u^{w-1}\dee u$ be the Mellin transform of $g$. Then
\als{
\sum_{n=1}^\infty \frac{g(n/N)\chi(n)}{n^{s_0}}
	&= \frac{1}{2\pi i} \int_{\Re(w)=3/4} L(s_0+w,\chi) N^w \hat{g}(w) \dee w \\
	&= \frac{1}{2\pi i} \int_{\Re(w)=-3/4} L(s_0+w,\chi) N^w \hat{g}(w) \dee w ,
}
by Cauchy's theorem, since $\hat{g}$ is entire by our assumption that $g$ is supported on $[1,4]$. We make the change of variable $s=1-w-s_0=\bar{s_0}-w$ so that
\[
\sum_{n=1}^\infty \frac{g(n/N)\chi(n)}{n^{s_0}}
	= \frac{N^{\bar{s_0}}}{2\pi i} \int_{\Re(s)=5/4} L(1-s,\chi) N^{-s}\hat{g}(\bar{s_0}-s) \dee s  .
\]
There is a complex number $\epsilon_\chi$ of modulus 1 such that
\[
L(1-s,\chi) = \frac{2\epsilon_\chi }{q^{1/2}} \left(\frac{q}{2\pi}\right)^{s} \gamma(s,\chi) L(s,\bar{\chi}), 
\]
where
\[
\gamma(s,\chi) = \Gamma(s) \cos\left(\frac{\pi(s-a)}{2}\right) 
\]
with $a=(1-\chi(-1))/2$. Therefore
\als{
\sum_{n=1}^\infty \frac{g(n/N)\chi(n)}{n^{1/2+it}} 
	&= \frac{\epsilon_\chi N^{\bar{s_0}} }{q^{1/2} \pi i } 
		\int_{\Re(s)=5/4} \left(\frac{q}{2\pi N}\right)^s \gamma(s,\chi)
			L(s,\bar{\chi})  \hat{g}(\bar{s_0}-s)  \dee s.
}
We expand the sum $L(s,\bar{\chi})$ and invert the order of integration and summation to find that
\als{
\sum_{n=1}^\infty \frac{g(n/N)\chi(n)}{n^{1/2+it}} 
	&= \frac{\epsilon_\chi N^{\bar{s_0}} }{q^{1/2}\pi i} 
			\sum_{n=1}^\infty \bar{\chi}(n)
			\int_{\Re(s)=5/4} \left(\frac{q}{2\pi nN}\right)^s \gamma(s,\chi)
				\hat{g}(\bar{s_0}-s)  \dee s.
}
We will show that the terms with $n>M_0:=(QT/N)^{1+\delta}$ contribute very little to the above sum. Indeed, for such an $n$, we shift the line of integration to the line $\Re(s)=A$. If $s=A+iu$, then we have that $\gamma(s,\chi)\ll_{A}1+|u|^{A-1/2}$ by Stirling's formula. Since $\hat{g}(A+iu)\ll_B 1/(1+|u|^B)$, for any $B>0$, we conclude that
\als{
\int_{\Re(s)=5/4} \left(\frac{q}{2\pi nN}\right)^s \gamma(s,\chi)
				\hat{g}(\bar{s_0}-s)  \dee s
	&=\int_{\Re(s)=A} \left(\frac{q}{2\pi nN}\right)^s \gamma(s,\chi)
				\hat{g}(\bar{s_0}-s)  \dee s \\
	&\ll_A\left(\frac{q}{nN}\right)^A \int_{-\infty}^{\infty} \frac{1+|u|^{A-1/2}}{1+|u-t|^B} \dee u  
		\ll_A \frac{1}{\sqrt{T}} \left(\frac{QT}{nN}\right)^A 
}
by taking $B=A+2$ and using our assumptions that $|t|\le T$ and that $q\le Q$. If $A\ge2/\delta$, then we conclude that
\als{
\sum_{n=1}^\infty \frac{g(n/N)\chi(n)}{n^{1/2+it}} 
	&= \frac{\epsilon_\chi N^{\bar{s_0}} }{q^{1/2}\pi i} 
			\sum_{n\le M_0} \bar{\chi}(n) 
			\int_{\Re(s)=5/4} \left(\frac{q}{2\pi nN}\right)^s \gamma(s,\chi)
				\hat{g}(\bar{s_0}-s)  \dee s
				+  O_{\delta,A}\left(  M^{-\delta A/2} \right) .
}
For the integers $n\le M_0$, we set $s=\bar{s_0}+w$, move the line of integration to the line $\Re(w)=0$ and invert the order of summation and integration to conclude that
\als{
\sum_{n=1}^\infty \frac{g(n/N)\chi(n)}{n^{1/2+it}} 
	&= \frac{\epsilon_\chi }{(2\pi)^{\bar{s_0}}q^{it}\pi i} 
		\int_{\Re(w)=0} \left(\frac{q}{2\pi N}\right)^{w}\gamma(\bar{s_0}+w,\chi) 
			\sum_{n\le M_0} \frac{\bar{\chi}(n)}{n^{\bar{s_0}+w}}  
				\hat{g}(-w)  \dee w \\
	&\qquad	+ O_{\delta,A}\left( M^{-\delta A/2} \right) .
}
If $w=-iu$, then $\gamma(\bar{s_0}+w,\chi)  \hat{g}(-w) \ll 1/(1+u^2)$. Therefore
\[
\sum_{n=1}^\infty \frac{g(n/N)\chi(n)(\log n)^r}{n^{1/2+it}} 
	\ll_{\delta,A}
			\int_{-\infty}^{\infty} 
				\left|\sum_{n\le M_0} \frac{\chi(n)}{n^{1/2+i(u+t)}}\right|
				\frac{\dee u}{1+u^2}  
		+ \frac{1}{M^{\delta A/2}} .
\]
If $M_0<1$, so that $M=1$, then the lemma follows immediately by the above estimate. If $M_0\ge1$, so that $M=M_0$, then the second term can be absorbed into the main term by taking $A$ large enough: we have that
\[
\int_{-M^2}^{M^2}\left|\sum_{n\le M}  \frac{\chi(n)}{n^{1/2+i(u+t)}} \right|^2 du
	\gg M^2
\]
by Theorem 9.1 in \cite{IK;2004}. Therefore
\[
\int_{-\infty}^{\infty} \left|\sum_{n\le M}  \frac{\chi(n)}{n^{1/2+i(u+t)}} \right| \frac{du}{1+u^2}
	\gg \frac{1}{M^{9/2}} \int_{-M^2}^{M^2}\left|\sum_{n\le M}  \frac{\chi(n)}{n^{1/2+i(u+t)}} \right|^2 du
	\gg \frac{1}{M^{5/2}} ,
\]		
which proves the lemma in the case when $M_0\ge1$ too by taking $A=5/\delta$.
\end{proof}


\section{The proof of Theorem \ref{mainthm2}}\label{mainthm2-proof}

Let $A,\epsilon$ and $x,h, Q$ be as in the statement of Theorem \ref{mainthm2}. All implied constants might depend on $\epsilon$ and $A$, as well as on the parameters $k_0,B,C$ and $\delta$, and the function $g$ introduced below.

Arguing as in Section \ref{mainthm1-proof}, we note that is enough to prove that
\[
\int_x^{4x} \sum_{D<d\le 2D} \ssum{\chi\mod d} S(y,\eta y;\Lambda\chi) \dee y
 		\ll \frac{D\eta x^2}{\CL^{A+3}}   \qquad(1\le D\le Q) ,
\]
where 
\[
\eta := \frac{h}{x\CL^{A+1}} =  x^{\theta-1}\CL^{-A-1} .
\]
(Here we don't need to remove the exceptional characters, even though we could do so using relation \eqref{zero-density} with $c=12/5+\epsilon$.) So from now on we fix $D\in[1,Q]$ and we set
\eq{beta-def}{
D^2 = \frac{\eta x}{x^\beta} = x^{\theta-\beta} \CL^{-A-1} ,
}
so that $\beta\in[\alpha+\epsilon/2,\theta]$. 

Next, we fix $k_0\in\N$ to be chosen later and use Heath-Brown's identity as in \cite[p.1367]{H-B;1982} to replace the von Mangoldt function by certain convolutions. Indeed, we have that
\[
\sum_{y<n\le y+\eta y} \Lambda(n)\chi(n)
	= \sum_{k=1}^{k_0} (-1)^{k-1}\binom{k_0}{k} 
		\sum_{\substack{y<n\le y+\eta y \\ n_{k+1},\dots,n_{2k}\le(4x)^{1/k_0} }}  (\log n_1)\mu(n_{k+1})\cdots \mu(n_{2k}) \chi(n_1\cdots n_{2k}) 
\]
for all $y\le 3x$. We break the range of $n_{k+1},\dots,n_{2k}$ into dyadic intervals $(N_{k+1},2N_{k+1}]$, \dots, $(N_{2k},2N_{2k}]$. For the range of $n_1,\dots,n_{2k}$ we will be more careful and use a smooth partition of unity: there is a smooth function $g:\R_{\ge0}\to\R_{\ge0}$ supported on $[1,4]$ such that
\[
\sum_{j\in\Z} g\left(\frac{x}{2^j}\right)    = 1  \qquad(x>0).
\]
Indeed, such a function can be constructed by fixing a smooth function $\tilde{g}:\R_{\ge0}\to\R_{\ge0}$ supported on $[1,2]$ such that $\int_0^\infty \tilde{g}(u) \dee u/u=1$, and by setting
\[
g(x) = \int_1^2 \tilde{g}(x/u) \frac{\dee u}{u} .
\]
Then we find that $S(y,\eta y;\Lambda\chi)$ can be bounded by $O(\CL^{2k_0})$ sums of the form $S(y,\eta y;f\chi)$, where $f=f_1*f_2*\cdots*f_{2k}$ for some $k\in\{1,\dots,k_0\}$ with 
\eq{f_j}{
f_j(n) = \begin{cases} 
		g(n/N_j) \log n		&\text{if}\ j=1,\\
		g(n/N_j) 			&\text{if}\ 2\le j\le k,\\
		{\bf 1}_{(N_j,2N_j]}(n) \mu(n) &\text{if}\ k+1\le j\le 2k,
	\end{cases}
}
and $N_1,\dots,N_{2k}$ being numbers that belong to $[1/4,4x]$ and satisfy the inequalities 
\[
N_{k+1},\dots,N_{2k}\le 4x^{1/{k_0}}
\quad\text{and}\quad
N_1\cdots N_{2k} \asymp x.
\] 
So it suffices to show that 
\eq{mainthm2-alt1}{
\int_x^{4x} \sum_{D<d\le 2D} \ssum{\chi\mod d} S(y,\eta y;f\chi)  \dee y 	
		\ll \frac{D\eta x^2}{\CL^{A+2k_0+3}} .
}
In order to detect the condition $y<n\le y+\eta y$ in $S(y,\eta y;f\chi)$, we use Perron's formula: if we set
\[
F_j(s,\chi) = \sum_{n=1}^\infty\frac{f_j(n) \chi(n) }{n^s} \quad(1\le j\le 2k)
\] 
and
\[
F(s,\chi) = F_1(s,\chi)\cdots F_{2k}(s,\chi) = \sum_{n=1}^\infty \frac{f(n)\chi(n)}{n^s},
\]
and we fix some $T_0\in(x/2,x]$, then using the lemma in \cite[p. 105]{Dav;2000} we find that
\eq{e70}{
\sum_{n\le z}f(n)\chi(n)
	= \frac1{2\pi i}  \int\limits_{\sub{\Re(s) = 1/2\\|\Im(s)|\le T_0}}  F(s,\chi) \frac{z^s}{s}  \dee s  
		+ O(x^{\epsilon/10})  \qquad(1\le z\le 5x) .
}
Applying this estimate for $z=y$ and $z=y+\eta y$ with $y\in[x,4x]$, we find that
\al{
	\int_x^{4x}  \sum_{d\le D}\ \ssum{\chi\mod d}   S(y,\eta;f\chi)   \dee y  
		&=   \frac1{2\pi}   \sum_{d\le D}\ \ssum{\chi\mod d}
				\int_x^{4x} \abs{ \int\limits_{\sub{\Re(s)=1/2 \\ |\Im(s)| \le T_0 }}  F(s,\chi)
					\frac{(1+\eta)^s-1}{s}\cdot y^s  \dee s } \nn
		&\quad			+ O\left(D^2x^{1+\epsilon/10}\right).   \label{e80}
}
Dividing the range of integration into $O(\CL)$ subsets of the form $\{s=1/2+it:T-1\le |t| \le 2T-1\}$, $T=2^m\ge1$, and choosing $T_0$ as the unique number of the form $2^m-1$ belonging to $(x/2,x]$, we find that \eqref{mainthm2-alt1} is reduced to showing that
\eq{mainthm2-alt2}{
 \sum_{d\le D}\ \ssum{\chi\mod d}
				\int_x^{4x} \abs{\ \int\limits_{\sub{\Re(s)=1/2 \\ T\le|\Im(s)|+1\le2T }}  F(s,\chi)
					\frac{(1+\eta)^s-1}{s}\cdot y^s  \dee s } 
		\ll \frac{D\eta x^2}{\CL^{A+2k_0+4}} 
}
for all $T\in[1,x/2]$. 

We continue by dividing the range of integration according to the size of the Dirichlet polynomials $F_j(s,\chi)$. To this end, we fix some numbers $U_1,\dots,U_{2k}$ with $1\le U_1\ll\sqrt{N_1}\log N_1$ and $1\le U_j\ll \sqrt{N_j}$ for $j\in\{2,3,\dots,2k\}$, and we set
\[
\CP(\chi,T,\bs U) = \{t\in\R: T\le|t|+1\le2T,\ U_j\le |F_j(1/2+t,\chi)|+1\le 2U_j\ (1\le j\le 2k)\} .
\]
Then relation \eqref{mainthm2-alt2} is reduced to showing that
\eq{mainthm2-alt3}{
 \sum_{D<d\le 2D}\ \ssum{\chi\mod d}
				\int_x^{4x} \abs{\,\int\limits_{\sub{s=1/2+it \\ t\in \CP(\chi,T,\bs U)} }  F(s,\chi)
					\frac{(1+\eta)^s-1}{s}\cdot y^s  \dee s } 
		\ll \frac{D\eta x^2}{\CL^{A+4k_0+4}} ,
}
for all $U_1,\dots,U_{2k}$ as above. From now on we fix such a choice of $U_1,\dots,U_{2k}$, and we set 
\[
U=U_1\cdots U_{2k} .
\]
Observe that $|F(1/2+it,\chi)|\asymp U$ for $t\in \CP(\chi,T,\bs U)$.  We fix two large enough constants $B$ and $C$ to be chosen later and we claim that we may assume that
\eq{U}{
U\le \min\{\CL^B\sqrt{x}/D,\sqrt{x}/\CL^C \} .
}
First, we show that we may assume that $U\le \CL^B\sqrt{x}/D$. Indeed, if $U > \CL^B\sqrt{x}/D$, then
\[
\int\limits_{\sub{s=1/2+it \\ t\in \CP(\chi,T,\bs U)} } F(s,\chi) \frac{(1+\eta)^s-1}{s}\cdot y^s\dee s 
	\ll \frac{\min\{\eta,1/T\}\sqrt{x}}{\sqrt{x}\CL^B/D}\int_{-2T}^{2T} |F(1/2+it, \chi)|^2\dee t.
\]
Consequently, Lemma\ \ref{large sieve} implies that
\als{
	\sum_{D<d\le 2D}&\ \ssum{\chi\mod d}
		\int_x^{4x}   \abs{ \,\int\limits_{\sub{s=1/2+it \\ t\in \CP(\chi,T,\bs U)}} 
			F(s,\chi)\frac{(1+\eta)^s-1}{s}\cdot y^s  \dee s }  \\
		&\ll   \frac{Dx\min\{\eta,1/T\}}{\CL^B}\sum_{d\le D}\ \sideset{}{^*}
			\sum_{\chi\mod d}\int_{-2T}^{2T}   |F(1/2+it,\chi)|^2\dee t  \\
		&\ll \frac{Dx\min\{\eta,1/T \}}{\CL^B}(D^2T+x)  \CL^{4k^2+2} 
			\ll  \frac{Dx}{\CL^{B-4k^2-2}}(D^2+\eta x)
			\ll  \frac{D\eta x^2}{\CL^{B-4k^2-2}},
}
which is admissible provided that $B\ge A+8k_0^2+6$, a condition we assume from now on. So we only need to consider the case when $U\le \CL^B\sqrt{x}/D$. 

Finally, we prove that we may restrict our attention to the case when $U\le \sqrt{x}/\CL^C$. If $D>\CL^{B+C}$, this is implied by our assumption that $U\le \CL^B\sqrt{x}/D$. So we assume that $D\le \CL^{B+C}$. We fix a small positive constant $\delta$ to be chosen later and we set
\[
\CJ=\{1\le j\le 2k: N_j>x^{\delta^2}\} 
\quad\text{and}\quad J=|\CJ|.
\]
As long as $\delta^2<1/(2k)$ (which we shall assume), we have that $J\ge1$. If $j\in\CJ$ and $\chi$ is a primitive character modulo some $d\in(D,2D]$, then we have that $|F_j(1/2+it,\chi)|\ll \sqrt{N_j}/\CL^{C+1}$ for all $t\in[-x,x]$. Showing this inequality is routine: we start by using Perron's formula to express $F_j(1/2+it,\chi)$ in terms of $L'(s,\chi), L(s,\chi)$ or $(1/L)(s,\chi)$, according to whether $j=1$, $1<j\le k$ or $k<j\le2k$, respectively. Then we shift the contour to the left and bound $L(s,\chi),L'(s,\chi)$ or $(1/L)(s,\chi)$ in the neighbourhood of the line $\Re(s)=1$ using exponential sum estimates due to Vinogradov. (Note that when $\Im(s)$ is small and $k<j\le 2k$, we need to make use of our assumption that $D\le\CL^{B+C}$ and to apply Siegel's theorem.) Since $|F_j(1/2+it,\chi)| \ll  \sqrt{N_j}/\CL^{C+1}$ for $t\in[-x,x]$, we may assume that $U_j\ll \sqrt{N_j}/\CL^{C+1}$ for all $j\in\CJ$, which in turn implies that $U=\prod_{j=1}^{2k}\ll \CL^{-C-1}\prod_{j=1}^{2k} \sqrt{N_j}\asymp \sqrt{x}/\CL^{C+1}$. So if $x$ is large enough, then $U\le \sqrt{x}/\CL^C$, as claimed.

Hence from now on we assume that $U$ satisfies relation \eqref{U}. Fix for the moment a character $\chi$ and consider the integral
\[
I=I(\chi,T,\bs U) = \int_{x}^{3x} \abs{\ \int\limits_{\sub{s=1/2+it \\ t\in \CP(\chi,T,\bs U)} } 
		F(s,\chi) \frac{(1+\eta)^s-1}{s} \cdot y^s \dee s } \dee y.
\]
Employing the Cauchy-Schwarz inequality, we find that
\als{
I^2
	&\le 2x\int_{x}^{3x} \abs{\ \int\limits_{\sub{s=1/2+it \\ t\in \CP(\chi,T,\bs U)} } 
		F(s,\chi) \frac{(1+\eta)^s-1}{s} \cdot y^s \dee s }^2 \dee y \\
	&= 2x\int_{x}^{3x}  \int\limits_{\sub{s_1=1/2+it_1 \\ t_1\in \CP(\chi,T,\bs U)} } 
		\int\limits_{\sub{s_2=1/2+it \\ t_2\in \CP(\chi,T,\bs U)} } 
		F(s_1,\chi)\overline{F(s_2,\chi)} \frac{(1+\eta)^{s_1}-1}{s_1} 
		\frac{(1+\eta)^{\overline{s_2}}-1}{\overline{s_2}} y^{s_1+\overline{s_2}}\dee s_1\dee s_2\dee y.
}
We first integrate over $y$ and then observe that 
\[
\frac{(1+\eta)^{1/2+it}-1}{1/2+it}\ll  \min\left\{\eta,\frac{1}{1+|t|} \right\}
\]
and that $|F(s_1,\chi)\overline{F(s_2,\chi)}|\le |F(s_1,\chi)|^2 +|F(s_2,\chi)|^2$. So we deduce that
\als{
	I^2 \ll x^3 \min\left\{\eta,\frac{1}{T} \right\}^2 \int_{\CP(\chi,T,\bs U)}   |F(1/2+it_1,\chi)|^2
		\int_{-2T}^{2T} \frac1{1+|t_1-t_2|}\dee t_2\dee t_1.
}
The inner integral is $\ll\log(2T)\ll\CL$, which reduces \eqref{mainthm2-alt3} to showing that
\[
 \sum_{D<d\le 2D}\ \ssum{\chi\mod d} \left( \int_{\CP(\chi,T,\bs U)}   |F(1/2+it,\chi)|^2 \dee t\right)^{1/2} 
		\ll  \frac{D\sqrt{x}\max\{1,\eta T\}}{\CL^{A+4k_0+9/2}} .
\]
By an application of the Cauchy-Schwarz inequality, we find that it is enough to show that
\eq{mainthm2-alt4}{
S(T,\bs U):= \sum_{D<d\le 2D}\ \ssum{\chi\mod d} \int_{\CP(\chi,T,\bs U)}   |F(1/2+it,\chi)|^2 \dee t
		\ll  \frac{\max\{1,\eta T\}^2  x}{\CL^{2A+8k_0+9}},
}
for all $T\in[1,x]$ and all $U_1,\dots,U_{2k}$ with $U\le \min\{\CL^B\sqrt{x}/D,\sqrt{x}/\CL^{C}\}$. We note that this relation follows immediately by Lemma \ref{large sieve} if $T>\CL^{A+6k_0^2+6}/\eta$. So from now on we will be assuming that $T\le \CL^{A+6k_0^2+6}/\eta$. Moreover, we set 
\eq{H-def}{
H=D^2T \le \CL^{A+6k_0^2+6}D^2/\eta = x^{1-\beta} \CL^{A+6k_0^2+6} ,
}
by the definition of $\beta$ by relation \eqref{beta-def}.

In order to show\ \eqref{mainthm2-alt4}, we will take advantage of the special product structure of $F(s,\chi)$, stemming from the fact that $f$ is a convolution. First, we consider the case when there is some $j\in\{1,\dots,k\}$ with $N_j>(DT)^{1/2+\delta}\CL^B$. For simplicity, let us assume that $j\ge2$, the argument when $j=1$ being similar. If $\delta_1$ is chosen so that $(1+\delta_1)/(2+\delta_1)=1/2+\delta$ and $M=\max\{1,(4DT/N_j)^{1+\delta_1}\}$, then Lemma \ref{charsums} implies that
\[
F_j(1/2+it,\chi) = \sum_{n=1}^\infty \frac{g(n/N_j)\chi(n)}{n^{1/2+it}} 
	\ll   \int_{-\infty}^{\infty} 
				\left|\sum_{n\le M } \frac{ \chi(n)} {n^{1/2+i(u+t)}}\right|
				\frac{\dee u}{1+u^2}  .
\]
Together with the Cauchy-Schwarz inequality and Lemma \ref{large sieve}, this implies that
\als{
S(T,\bs U)
	&\ll \int_{-\infty}^{\infty}  \sum_{D<d\le 2D}\ssum{\chi\mod d} 
		\int_{-2T}^{2T} 
				\left|\sum_{n\le M } \frac{\chi(n)}{n^{1/2+i(u+t)}}\right| ^2
					\prod_{\substack{1\le \ell\le 2k \\ \ell\neq j}} |F_\ell(1/2+it,\chi)|^2 \frac{ \dee t\dee u}{1+u^2} \\
	&\ll  \left(\frac{x+x(DT/N_j)^{1+\delta_1}}{N_j} + H \right) \CL^{4k^2+2} 
		\le \left(\frac{2x}{\CL^{B}} + H \right) \CL^{4k^2+2} .
}
Since $H\le x^{1-\beta+o(1)}$ by \eqref{H-def}, we conclude that $S(T,\bs U) \ll x/\CL^{B-4k^2-2}$, so \eqref{mainthm2-alt4} does hold, provided that $B\ge 2A+12k_0^2+11$.

The above discussion allows to assume that if $j\in\{1,\dots,k\}$, then 
\[
N_j\le (DT)^{1/2+\delta}\CL^B \le x^{(1/2+\delta)(1-(\theta+\beta)/2)+o(1)}\le x^{1/2-\delta},
\]  
provided that $\delta$ is small enough. We suppose that $k_0\ge 3$, so that $N_j\le x^{1/3}$ for all $j\in\{k+1,\dots,2k\}$. Therefore, we see that $J=|\CJ|\ge 3$. 

Next, we reduce \eqref{mainthm2-alt4} to a problem about large values of Dirichlet polynomials. We set
\[
\CZ(\chi,T,\bs U) = \left\{n\in\Z:  [n,n+1]\cap \CP(\chi,T,\bs U) \neq \emptyset \right\} =: \{n_1,\dots,n_r\},
\]
say, with $n_1<n_2<\cdots<n_r$. For each $j\in\{1,\dots,r\}$, we select $t_j\in[n_j,n_{j+1}]\cap \CP(\chi,T,\bs U)$, and we set 
$\CR_m(\chi,T,\bs U)=\{t_j:1\le j\le r,\ j\equiv m\mod{2}\}$ for $m\in\{0,1\}$, then the sets
\[
\CR_m(T,\bs U):= \bigcup_{D<d\le 2D} \bigcup_{\substack{\chi\mod d \\ \chi\ \text{primitive} }} 
	\{(\chi,t): t\in \CR_m(\chi,T,\bs U)\} \qquad(m\in\{0,1\})
\] 
are well-spaced according to the definition before Lemma \ref{large values}. Finally, we note that
\[
S(T,\bs U) \ll  U^2 (|\CR_0(T,\bs U)|+|\CR_1(T,\bs U)|),
\]
which reduces \eqref{mainthm2-alt4} to showing that
\eq{mainthm2-alt5}{
|\CR_m(T,\bs U)|  \ll \frac{x}{\CL^{2A+8k_0+9} U^2}  \quad(m\in\{0,1\}) .
}
We fix $m\in\{0,1\}$ and proceed to the proof of \eqref{mainthm2-alt5}. We distinguish several cases. 

\medskip

\underline{\textsc{Case 1}}. Assume that there is a $j\in \CJ$ such that $U_j>\sqrt{N_j}/\CL^B$.

\medskip

\noindent  We fix a large enough positive integer $r$ so that $U_j^{2r}\ge H$ and we apply Lemma \ref{large values} with $F_j(s,\chi)^r$ in place of $A(s,\chi)$ to deduce that
\[
|\CR_m(T,\bs U)| \ll \frac{N_j^r+H}{U_j^{2r}} \CL^{9r^2+18}   \ll \CL^{4rB+9r^2+18} 
	 \ll \frac{x}{U^2\CL^{2A+8k_0+9}}
\]
by \eqref{U}, provided that $C$ is large enough. So \eqref{mainthm2-alt5} does hold in this case.

\medskip

\underline{\textsc{Case 2}}. Suppose that $U_j\le \sqrt{N_j}/\CL^B$ for all $j\in \CJ$ and that there is some $j\in \CJ$ such that $U_j\le x^{\beta/2}/\CL^{B}$. 

\medskip

\noindent We apply Lemma \ref{large sieve} with $\prod_{\ell\neq j} F_\ell(s,\chi)$ in place of $A(s,\chi)$ and use relation \eqref{H-def} to deduce that
\als{
|\CR_m(T,\bs U)|\ll \frac{x/N_j+H}{(U/U_j)^2}  \CL^{12k^2+24} 
		&= \frac{x}{U^2} \left( \frac{U_j^2}{N_j} + \frac{U_j^2H}{x} \right) \CL^{12k^2+24}  \\
		&\ll \frac{x}{U^2} \left( \frac{1}{\CL^{2B}}  + \frac{(x^\beta/\CL^{2B}) x^{1-\beta}\CL^{A+6k_0^2+6}}{x} \right) 
			\CL^{12k^2+24} ,
}
which shows \eqref{mainthm2-alt5} by taking $B\ge 3A/2+13k_0^2+20$. 

\medskip

\underline{\textsc{Case 3}}. Assume that $U_j\in[x^{\beta/2}/\CL^B,\sqrt{N_j}/\CL^{2B}]$ for all $j\in \CJ$ and that $\beta\ge1/(2J)+\epsilon/2$, where $J=|\CJ|$. 

\medskip

\noindent  In this case we argue by contradiction: we begin by assuming that $|\CR_m(T,\bs U)|\ge x/(U^2\CL^{2A+8k_0+9})$. 
For every $j\in \CJ$, we apply Lemma \ref{large values} with $\prod_{\ell\neq j} F_\ell(s,\chi)$ in place of $A(s,\chi)$ to deduce that
\als{
\frac{x}{U^2\CL^{2A+8k_0+9}} 
 	\le |\CR_m(T,\bs U)|
		&\ll \left(\frac{x/N_j}{(U/U_j)^2} + \frac{Hx/N_j}{(U/U_j)^6}\right)\CL^{12k^2+24}  \\
	&\le \frac{x}{U^2\CL^{2B-12k^2-24}} + \frac{x^{2-\beta}U_j^6}{U^6N_j} \CL^{A+18k_0^2+30} ,
}
by \eqref{H-def} and our assumption on $U_j$. Since $B\ge 3A/2+13k_0^2+20$, this implies that
\[
\frac{x}{U^2\CL^{2A+8k_0+9}}  \ll	 \frac{x^{2-\beta}U_j^6}{U^6N_j} \CL^{A+18k_0^2+30} 
\qquad\implies\qquad 
\frac{U^4}{U_j^6} \ll	 \frac{x^{1-\beta}}{N_j} \CL^{3A+26k_0^2+39}  .
\]
We multiply the last inequality for all $j\in \CJ$ to find that
\[
U^{4J-6} \ll  x^{J(1-\beta)-1+k\delta^2} \CL^{J(3A+26k_0^2+39)} .
\]
We have that $U\ge\prod_{j\in \CJ}U_j\ge (x^{\beta/2}/\CL^{B})^J$. Therefore
\[
(2J-3)J\beta \le J-1-\beta J + (k+1)\delta^2
\qquad\implies\qquad 
2J\beta\le 1 + (k+1)\delta^2,
\]
which contradicts our assumption that $\beta\ge1/(2J)+\epsilon/2$ if $\delta$ is small enough in terms of $\epsilon$ and $k$. 
Therefore relation \eqref{mainthm2-alt5} holds in this case too.

\medskip

We note that since $J\ge3$, the above discussion shows Theorem \ref{mainthm1} when $\theta\le 1/2$. Indeed, in this case $\beta\ge \alpha+\epsilon/2=1/6+\epsilon/2\ge1/(2J)+\epsilon/2$, so we see that Cases 1-3 above are exhaustive. Hence from on we may assume that $\theta>1/2$. For such a $\theta$, we must have that $\alpha+\theta\ge 2/3$ by the definition of $\alpha$. Since $\beta\ge\alpha+\epsilon/2$, we must have that $\beta+\theta\ge2/3+\epsilon/2$. So the last case we consider is:

\medskip

\underline{\textsc{Case 4}}. Assume that $U_j\in[x^{\beta/2}/\CL^B,\sqrt{N_j}/\CL^{2B}]$ for all $j\in \CJ$, that $\beta<1/(2J)+\epsilon/2$, and that $\beta+\theta\ge2/3+\epsilon/2$. 

\medskip

\noindent Recall that $N_j\le (DT)^{1/2+\delta}=x^{(1/2+\delta)(1-(\theta+\beta)/2)+o(1)}$ for all $j\in\{1,\dots,k\}$. If $\delta$ is small enough in terms of $\epsilon$, then our assumption that $\beta+\theta\ge2/3+\epsilon/2$ implies that $N_j\le x^{1/3-\epsilon/10}$ when $1\le j\le k$. If we take $k_0=4$, so that $N_j\le x^{1/4}$ when $k+1\le j\le 2k$, then we must have that $J\ge 4$. In particular, $\beta<1/8+\epsilon/2$, whence we deduce that $\alpha<1/8$. The definition of $\alpha$ then implies that $\theta>5/8$, so that $\alpha=(1-\theta)/3$. Recall that $U\le \CL^B\sqrt{x}/D$. Since $U^2\ge (x^{\beta}/\CL^{2B})^J \ge x^{4\beta}\CL^{-8B}$, we conclude that
\[
x^{4\beta}\le \frac{\CL^{10B} x}{D^2} = \CL^{10B+A+1} x^{1-\theta+\beta}.
\]
Therefore $\beta\le(1-\theta)/3-o(1)$. However, this contradicts the fact that $\beta\ge \alpha+\epsilon/2\ge (1-\theta)/3+\epsilon/2$. This shows that this last case cannot actually occur, thus completing the proof of relation \eqref{mainthm2-alt5} and hence of Theorem \ref{mainthm2}.


\begin{rem}
The above method cannot be improved without some additional input. This can be seen by taking $N_j=x^{1/J}$ when $1\le j\le J$ and $N_j=1$ when $J<j\le 2k_0$, so that $\CJ=\{1,\dots,J\}$. Also, we assume that $U_j=N_j^{1/4}$, so that when we apply Lemma \ref{large values}, the two different expressions on the right hand side balance. Then the only estimate we can extract from Lemma \ref{large values} is 
\[
|\CR_m(T,\bs U)|\ll \left(\frac{x^{r/J}}{U^{2r/J}} + \frac{H}{U^{2r/J}}\right)\CL^{O(1)}  \quad(r\in\N) .
\]
The largest $r$ we can take while still having that $x^{r/J}/U^{2r/J}$ is smaller than $x/U^2$ is $r=J-1$. For this choice, assuming also that $T=1/\eta$, we find that $H/U^{2r/J} = U^{2/J}H/U^2 = U^{2/J}x^{1-\beta}/U^2$. So in order to make this expression smaller than $x/U^2$, we need to assume that $U^2\le x^{J\beta}$. Since $U_j=N_j^{1/4}$, we have that $U^2\asymp x^{1/2}$, so we must have that $\beta\ge1/(2J)$. As it is clear from the proof, what allows us improve upon this estimate when $\theta>5/8$ is the fact that we also know that $U\le\CL^B\sqrt{x}/D$.
\end{rem}


\section{Proof of Theorem \ref{mainthm3}}\label{mainthm3-proof}

Fix $\epsilon,A$ and $x,h,Q$ as in the statement of Theorem \ref{mainthm3}. As in the previous section, all implied constants might depend on $\epsilon$ and $A$. 

When $h>x/\CL^{A+1}$, then Theorem \ref{mainthm3} follows by the Bombieri-Vinogradov theorem. So assume that $h\le x/\CL^{A+1}$. Note that we also have that $h\ge x^{1/15+\epsilon}$. Therefore, using Buchstab's identity as in \cite{Li;1997}, we can construct a function $\rho:\N\to\R$ that satisfies the following properties: 
\begin{itemize}
\item $\rho\le  {\bf 1}_{\mathbb P}$, where ${\bf 1}_{\mathbb P}$ is the indicator function of the set of primes;
\item $\rho$ is supported on integers free of prime factors $<x^{1/100}$;
\item $\rho(n)=O(1)$ for all $n\le 2x$; 
\item there is a positive constant $c_1$ such that
\eq{rho}{
\int_x^{2x}  \abs{ \sum_{ y<n \le y+h  } \rho(n) - \frac{c_1 h}{\log y} } \dee y
		\ll \frac{hx}{\CL^{A+2}}  
}
\end{itemize}
We claim that $\rho$ also satisfies the inequality
\eq{mainthm3-alt1}{
\int_x^{2x}\sum_{q\le Q}
	\max_{(a,q)=1}  \abs{ \sum_{\substack{ y<n \le y+h \\ n\equiv a\mod q} } \rho(n) - \frac{c_1 h}{\phi(q)\log y} } \dee y
		\ll_{A,\epsilon} \frac{hx}{\CL^A}  ,
}
for all $A>0$. Before proving this relation, we will demonstrate that it implies Theorem \ref{mainthm3}. Without loss of generality, we may assume that $h\in\Z$. Consider integers $x_1,x_2,Q_1,Q_2$ with $x\le x_1<x_2\le 2x$ and $Q/2\le Q_1<Q_2\le Q$. Then relation \eqref{mainthm3-alt1} immediately implies that
\[
\sum_{x_1\le m < x_2}\sum_{Q_1<q\le Q_2}
	\max_{(a,q)=1}  \abs{ \sum_{\substack{ m<n \le m+h \\ n\equiv a\mod q} } \rho(n) - \frac{c_1 h}{\phi(q)\log m} } 
		\ll  \frac{hx}{\CL^A}  .
\]
Therefore, the number of pairs $(q,m)\in\N^2\cap((Q_1,Q_2]\times[x_1,x_2))$ such that
\[
\abs{ \sum_{\substack{ m<n \le m+h \\ n\equiv a\mod q} } \rho(n) - \frac{c_0h}{\phi(q)\log m} } 
	\ge \frac{c_1h}{2\phi(q)\log y} 
\]
is $O(Qx\CL^{-A+1})$. Since $\rho\le {\bf 1}_{\mathbb P}$ and $\log n\sim \log m$ for $n\in[m,m+h]\subset[x,3x]$, we deduce that
\als{
\#\left\{  \begin{array}{c} (q,m)\in \N^2 \\ Q_1<q\le Q_2 \\ x_1\le m< x_2 \end{array} : 
\sum_{\substack{m<p\le m+h \\ p\equiv a\mod q}} \log p \ge \frac{c_1h}{3\phi(q)}\  \text{when}\ (a,q)=1\right\} 
	&= (Q_2-Q_1)(x_2-x_1) \\
	&\quad + O\left(\frac{Qx}{\CL^{A-1}}\right)   .
}
Theorem \ref{mainthm3} then follows with $c=c_1/3$ by the above estimate and a dyadic decomposition argument, after replacing $A$ with $A+1$.

So we have reduced our task to showing relation \eqref{mainthm3-alt1}. In view of relation \eqref{rho}, we further note that we may instead show that
\eq{mainthm3-alt2}{
\int_x^{2x}\sum_{q\le Q}
	\max_{(a,q)=1}  \abs{ \sum_{\substack{ y<n \le y+h \\ n\equiv a\mod q} } \rho(n) - \frac{1}{\phi(q)}
		\sum_{ y<n \le y+h  } \rho(n) } \dee y
		\ll   \frac{hx}{\CL^A}  .
}
Hence, arguing as in Section \ref{mainthm1-proof} and using the zero-density estimate \eqref{zero-density} with $c=12/5+\epsilon$, we find that it suffices to prove that
\eq{mainthm3-alt3}{
\int_x^{4x}\sum_{\substack{D<d\le 2D \\ d\notin\CE}}\ssum{\chi\mod d}
	S(y,\eta y;\rho\chi) \dee y
		\ll    \frac{D\eta x^2}{\CL^{A+4}} 
}
with $\eta= h/(x\CL^{A+1})$ and $\mathcal{E}$ being the set of moduli $q\le Q$ which are multiples of integers $d$ modulo which there exists a primitive Dirichlet character $\chi$ such that $N(\sigma_0,x,\chi)\ge1$, where $\sigma_0= 1- c_0(\log\CL)/\CL$ for some constant $c_0$ that can be taken arbitrarily large. There is only one thing that does not transfer immediately: relation \eqref{(n,q)>1}. Indeed, we need to be more careful when passing from all Dirichlet characters to primitive ones. Note that in order to perform this passage we need to show that
\[
S:= \sum_{q\le Q} \frac{1}{\phi(q)} \sum_{d|q} \ssum{\chi \mod d} \abs{ \sum_{\substack{y<n\le y+\eta y \\ (n,q)>1}} \chi(n)\rho(n) } \ll \frac{\eta y}{\CL^A},
\]
for all $y\in[x,3x]$. We write $d=qm$ and note that the the presence of the character $\chi$ in the inner sum implies automatically that the sum is supported on integers $n$ which are coprime to $d$. So the condition $(n,q)>1$ may be replaced by the condition $(n,m)>1$. Splitting also the range of $d$ into $O(\CL)$ intervals of the form $(D,2D]$, we deduce that
\[
S \ll	\CL^2 \max_{1/2\le D\le Q} \frac{1}{D} \sum_{m\le Q/D}\frac{1}{m} \sum_{D<d\le 2D} 
	\ssum{\chi \mod d} \abs{ \sum_{\substack{y<n\le y+\eta y \\ (n,m)>1}} \chi(n)\rho(n) } .
\]
We fix $D\in[1/2,Q]$ and $m\le Q/D$ and note that the Cauchy-Schwarz inequality and the Large Sieve \cite[p. 160, Theorem 4]{Dav;2000} imply that
\als{
\left(\frac{1}{D}  \sum_{D<d\le 2D} 
	\ssum{\chi \mod d} \abs{ \sum_{\substack{y<n\le y+\eta y \\ (n,m)>1}} \chi(n)\rho(n) } \right)^2
	&\ll \sum_{D<d\le 2D} 
		\ssum{\chi \mod d} \abs{ \sum_{\substack{y<n\le y+\eta y \\ (n,m)>1}} \chi(n)\rho(n) }^2 \\
	&\ll  \eta y\sum_{\substack{y<n\le y+\eta y \\ (n,m)>1}} |\rho(n)|^2 ,
}
since $D\le Q\le \sqrt{\eta y}$ by assumption. By the properties of $\rho$ mentioned above we find that
\als{
 \sum_{\substack{y<n\le y+\eta y \\ (n,m)>1}} |\rho(n)|^2 	
 	\ll  \sum_{\substack{y<n\le y+\eta y \\ p|n\ \Rightarrow\ p>x^{1/100} \\ (n,m)>1}} 1
	\le \sum_{\substack{g|m,\, g>1 \\ p|g\ \Rightarrow\ p>x^{1/100} }}
		\sum_{\substack{y<n\le y+\eta y \\ (n,m)=g}} 1 
	&\ll \sum_{\substack{g|m \\ g>x^{1/100}}}  \left(\frac{\eta y}{g} + 1\right) 
		\ll \frac{\tau(m)\eta y}{x^{1/100}}.
}
Consequently,
\[
S\ll \frac{\CL^2\eta y}{x^{1/200}} \sum_{m\le Q/D} \frac{\sqrt{\tau(m)}}{m}
	\ll \frac{\eta y}{\CL^A},
\]
as claimed. So we may indeed focus on proving relation \eqref{mainthm3-alt3}. 

Finally, arguing alone the lines of Section \ref{mainthm1-proof} and of Section \ref{mainthm2-proof}, we find that relation \eqref{mainthm3-alt3} can been reduced to showing that
\[
\sum_{\substack{D<d\le 2D \\ d\notin\CE}}\ssum{\chi\mod d} \int_{-2T}^{2T}
	\abs{\sum_{n\asymp x} \frac{\rho(n)\chi(n)}{n^{1/2+it}} }^2 \dee t \ll \frac{x}{\CL^{A_1}}   
			\quad(1\le T\le \eta^{-1}\CL^{A_1},\ D\le Q) ,
\]
for some constant $A_1$ that is sufficiently large in terms of $A$. The fact that $\rho$ does satisfy this bound is now a consequence of the methods in \cite{Li;1997}. Indeed, here the `length of the average' $D^2T$ (we are summing over about $D^2$ characters and integrating over an interval of length $T$) satisfies the inequality $D^2T\le x^{14/15+2\epsilon}$, which is precisely the inequality needed for $T$ in \cite{Li;1997}, which is the length of the average there. We can then employ Lemma \ref{large values} in a completely analogous way to the one Hal\'asz's method is used in \cite{Li;1997}. We also need an additional input: that, for all characters $\chi$ we are averaging over, and for all $N\ge x^\delta$, we have
\[
\sum_{N<p\le 2N} \frac{\chi(p)}{p^{1/2+it}} \ll \frac{\sqrt{N}}{\CL^{A_2}} ,
\]
where $A_2$ is a sufficiently large constant. This is a consequence of our assumption that $d>D\ge1$ (so $\chi$ is non-principal) and that $d\notin\CE$, provided that $c_0$ is large enough in terms of $A_2$. Thus Theorem \ref{mainthm3} follows. 

\begin{rem}
With a little more care, we may even assume that we work with primitive characters $\chi$ such that
\eq{improved prime estimate}{
\sum_{N<p\le 2N} \frac{\chi(p)}{p^{1/2+it}} \ll \frac{\sqrt{N}}{\exp\{\CL^{1/3-\epsilon}\}}.
}
Let $z=\exp\{\CL^{2/3+\epsilon/2}\}$ and let $\CE'$ be the set integers $d$ modulo which there exists a primitive Dirichlet character $\chi$ such that $N(1-c/\log z,x,\chi)\ge1$, where $c$ is a small enough constant. Page's theorem \cite[p. 95]{Dav;2000} and the Korobov-Vinogradov zero-free region for Dirichlet $L$-functions (see the notes of Chapter 9 in \cite{Mon;1994}) imply that $\CE'$ contains at most one element $\le z$. Then the argument leading to \eqref{exceptional} allows us to restrict our attention to $q\in[1,Q]$ which are not multiples of integers in $\CE'$. Moreover, if $\chi$ is a non-principal character modulo such a $q$, then relation \eqref{improved prime estimate} holds.
\end{rem}


\bibliographystyle{alpha}

\end{document}